\documentclass[11pt]{amsart}
\usepackage{hyperref,cleveref,xcolor}

\title[$(G,\ast)$-identities of $\mathrm{UT}_n$]{On star-homogeneous-graded polynomial identities of upper triangular matrices over an arbitrary field}
\author{Thiago Castilho de Mello}
\address{Instituto de Ciência e Tecnologia, Universidade Federal de S\~ao Paulo, SP, Brazil}
\email{tcmello@unifesp.br}
\author{Felipe Yukihide Yasumura}
\address{Department of Mathematics, Instituto de Matem\'atica e Estat\'istica, Universidade de S\~ao Paulo, SP, Brazil}
\email{fyyasumura@ime.usp.br}

\thanks{This work is supported by S\~ao Paulo Research Foundation (FAPESP), grants 2018/23690-6 and 2023/03922-8, and by CNPq, grants 405779/2023-2 and 404851/2021-5.}

\newtheorem{Thm}{Theorem}
\newtheorem{Lemma}[Thm]{Lemma}

\newtheorem{Cor}[Thm]{Corollary}
\theoremstyle{definition}
\newtheorem{Def}[Thm]{Definition}
\theoremstyle{remark}

\newtheorem*{Remark}{Remark}

\begin{document}
\begin{abstract}
We study the graded polynomial identities with a homogeneous involution on the algebra of upper triangular matrices endowed with a fine group grading. We compute their polynomial identities and a basis of the relatively free algebra, considering an arbitrary base field. We obtain the asymptotic behaviour of the codimension sequence when the characteristic of the base field is zero. As a consequence, we compute the exponent and the second exponent of the same algebra endowed with any group grading and any homogeneous involution.
\end{abstract}
\maketitle

\section{Introduction}
This paper concerns a family of graded algebras endowed with a compatible involution, their polynomial identities and codimension sequence. Namely, we are interested in the algebra of upper triangular matrices over a field $\mathbb{F}$, denoted by $\mathrm{UT}_n$. Their classical polynomial identities and related invariants are known (see, for instance, \cite{GiZabook}).

The study of graded polynomial identities appeared as a way to break the algebra into smaller pieces, making it easier to study its polynomial identities, in some sense. Several applications and examples appeared in this context, but the theory of graded polynomial identities much developed in recent years, and now it has acquired its importance on its own. In the context of the upper triangular matrix algebras, the group gradings are known \cite{VaZa2007}, as well as their graded polynomial identities \cite{VinKoVa2004,GRiva}, over an arbitrary field. Concerning a non-associative structure on the same vector space, it is interesting to highlight the difficulty to characterize their graded polynomial identities and related invariants, as studied in \cite{CMa,CMaS,CorK,DimasSa,GSK,KMa,PdrManu,Y23}.

The involutions of the first kind on $\mathrm{UT}_n$ were classified in \cite{VKS}. In the same paper, the authors classify the $\ast$-polynomial identities when $n\le3$ and the base field is infinite. Graded involutions are dealt in \cite{VaZa2007b} (see \cite{FDGY} as well). There are other works dedicated to study $\ast$-polynomial identities on $\mathrm{UT}_n$, for instance, \cite{UruG}, where the $\ast$-polynomial identities of $\mathrm{UT}_2$ is classified when the base field is finite. Also, the involutions of the second kind on $\mathrm{UT}_n$ were classified in \cite{UrureCouto}.

Compatibility of a grading and an involution is important and appeared in the classification of isomorphism classes of group gradings in important family of algebras, see for instance, \cite{BSZ, EK2013}. Several recent papers concern the context of (abelian) graded polynomial identities with a graded involution, for instance, \cite{BSOV21,CISV21,FSV16,GIM19,GSV16,GSV20,I20,IM17,ISSV,OSV23,OSV,OV21}.

On the other hand, it is natural to consider involutions that inverts the degrees. For instance, see \cite{Haz} and references therein. Graded polynomial identities with a degree-inverting involution are natural as well if we want to study these problems in the context of a non-abelian group. Concerning this case, the papers \cite{HN,MFb} study polynomial identities, while \cite{MF,FSY} classify degree-inverting involutions on graded-simple algebras.

Finally, it is natural to consider the following compatibility relation of an involution and a grading. We say that an involution is homogeneous if it sends a homogeneous component onto a homogeneous component. It was first considered in \cite{Mello}. The paper \cite{Y} concerns the study of graded polynomial identities where the algebra has a homogeneous involution.

In this paper, we consider the algebra $\mathrm{UT}_n$ endowed with a kind of a Universal grading (the unique fine grading of this algebra) and a homogeneous involution $\ast$. We compute the $(G,\ast)$-polynomial identities and a basis of its relatively free $G$-graded algebra with involution (\Cref{polid}). Then, we compute the asymptotic behaviour of the codimension sequence (\Cref{codimThm}), and in particular, we derive its exponent. As a consequence, for any group grading and homogeneous involution on $\mathrm{UT}_n$, we obtain the exponent and the second exponent of the algebra (\Cref{expsecexp}). As mentioned before, computing the $\ast$-polynomial identities of $\mathrm{UT}_n$ is a hard problem. However, considering a specific grading, where $\ast$ is homogeneous, the problem of studying the identities with involution becomes feasible. It is worth mentioning that in \cite{DiogoGaldino} the authors investigate the upper triangular matrix algebra endowed with a fine grading in the context of graded involutions. They compute the $\ast$-graded polynomial identities and asymptotic behaviour of the codimension sequence. Their results are independent of the results of the present paper since the context is distinct.

\section{Preliminaries}
\subsection{Graded algebra} Let $G$ be any group. We use the multiplicative notation for $G$ and denote its neutral element by $1$. We say that an algebra $\mathcal{A}$ is $G$-graded if there exists a vector-space decomposition $\mathcal{A}=\bigoplus_{g\in G}\mathcal{A}_g$ such that $\mathcal{A}_g\mathcal{A}_h\subseteq\mathcal{A}_{gh}$, for all $g,h\in G$. The choice of the decomposition is called a \emph{$G$-grading}, and we are going to denote it by $\Gamma$. The subspace $\mathcal{A}_g$ is called \emph{homogeneous component of degree $g$}. A nonzero element $x\in\mathcal{A}_g$ is called a homogeneous element of degree $g$ and we denote $\deg_\Gamma x=g$. The \emph{support} of the grading $\Gamma$ is $\mathrm{Supp}\,\Gamma=\{g\in G\mid\mathcal{A}_g\ne0\}$.

Finally, we provide a precise definition of the following:
\begin{Def}
Let $\mathcal{A}=\bigoplus_{g\in G}\mathcal{A}_g$ be a $G$-graded algebra, and let $\psi:G\to G$ be a map. An involution $\ast$ on $\mathcal{A}$ is a \emph{homogeneous involution} with respect to $\psi$ or a $\psi$-involution if $\mathcal{A}_g^\ast\subseteq\mathcal{A}_{\psi(g)}$, for all $g\in G$.
\end{Def}

We are specially interested in the case where the map $\psi$ is an anti-automorphism of order (at most) $2$ of the grading group. We shall usually denote the involution on $G$ by $\ast$ as well. So, we may write that $\ast$ is a $\ast$-homogeneous involution on $\mathcal{A}$.
\\

\noindent\textbf{Examples.}
\begin{enumerate}
\item If $G$ is an abelian group, then every degree-preserving involution is a homogeneous involution with respect to the identity map of $G$.
\item A {degree-inverting involution} is a homogeneous involution with respect to the inversion of $G$. It is worth mentioning that the degree-inverting involution on matrix algebras and upper triangular matrices were described in \cite{MF,FSY}.
\item Consider the $C_n\times C_n$-grading on $M_n(\mathbb{C})$, given by the following. Let $\varepsilon\in\mathbb{C}$ be a primitive $n$-th root of $1$. If $(i,j)\in C_n\times C_n$, then  $\mathcal{A}_{(i,j)}=\mathrm{Span}\{X^iY^j\}$, where
$$
X=\left( \begin{array}{cccc}
\varepsilon^{n-1}	& 0 & \cdots & 0 \\ 
0	& \varepsilon^{n-2} & \cdots & 0 \\ 
\vdots	& \vdots & \ddots & \vdots \\ 
0	& 0 & \cdots & 1
\end{array}\right),\quad Y=\left( \begin{array}{ccccc}
0	& 1 & 0 &\cdots & 0 \\ 
0	& 0 & 1 & \cdots & 0 \\ 
\vdots	& \vdots & \vdots & \ddots & \vdots \\ 
0	& 0 & 0 & \cdots & 1\\
1	& 0 & 0 & \cdots & 0
\end{array}\right).
$$
It is well known that such decomposition gives a $C_n\times C_n$-grading on $M_n(\mathbb{C})$. We denote such grading by $\Gamma_\varepsilon$. It is known that, if $\Gamma_\varepsilon$ is endowed with a degree-preserving or a degree-inverting involution, then $n=2$ (\cite[Lemma 2.50]{EK2013} and \cite[Lemma 5.6]{MF}). However, it is easy to see that the usual transposition will be a homogeneous involution for $\Gamma_\varepsilon$, for any $n\in\mathbb{N}$.
\end{enumerate}

\subsection{Free graded algebra with homogeneous involution}
We shall provide a construction of the free graded algebra endowed with a homogeneous involution. This is done using a particular case of the (relatively) free universal algebra in an adequate variety (see, for instance, \cite[Chapter 1]{Razmyslov} for a general discussion, and \cite{BY,BY2} as well for a particular graded version). Let $G$ be any group, and $X^G=\bigcup_{g\in G}X^{(g)}$, where $X^{(g)}=\{x_1^{(g)},x_2^{(g)},\ldots\}$. Let $\ast:G\to G$ be an involution, that is, an anti-automorphism of order (at most) $2$. 
Let $\mathbb{F}\{X^G,\ast\}$ denote the absolutely free $G$-graded binary algebra endowed with an unary operation (also denote by $\ast$).
We define the \emph{free $G$-graded associative algebra with a homogeneous involution} with respect to $\ast$, $\mathbb{F}\langle X^G,\ast\rangle$, 
as the quotient of $\mathbb{F}\{X^G,\ast\}$ by the following polynomials
\begin{align}\label{eqvariety}
\begin{split}
&x_1^{(g_1)}(x_2^{(g_2)}x_3^{(g_3)})-(x_1^{(g_1)}x_2^{(g_2)})x_3^{(g_3)}\\%
&(x_1^{(g_1)}x_2^{(g_2)})^\ast-(x_2^{(g_2)})^\ast(x_1^{(g_1)})^\ast\\%
&((x^{(g)})^\ast)^\ast-x^{(g)}\\%
&\deg_G(x^{(g)})^\ast - (\deg_Gx^{(g)})^\ast.
\end{split}
\end{align}
The first polynomial defines associativity while the second and third indicate that $\ast$ acts as an involution in the quotient algebra. The last one defines a relation between the involution $\ast$ of the group and the unary operation $\ast$ of the algebra. 

In fact, the last is natural in the context of graded polynomial identities with an involution, but to see this we need to describe the $G$-grading in terms of the projections (see, for instance, \cite{BY}). For each $g\in G$, let $\pi_g$ denote the unary operation on a $G$-graded algebra $\mathcal{A}$ given by the projection and inclusion $\pi_g:\mathcal{A}\to\mathcal{A}$. Then, the absolutely free $G$-graded algebra $\mathbb{F}\{X^G,\ast\}$ is a quotient of the absolutely free $\Omega$-algebra, where $\Omega$ contains one binary operation and $|G|+1$ unary operations (corresponding to each projection, and the involution). The quotient is given by the relations that define the $G$-grading, that is, $\pi_g(\pi_h(x))=\delta_{gh}\pi_h(x)$ and $\pi_g(\pi_{g_1}(x)\pi_{g_2}(y))=\delta_{g,g_1g_2}\pi_{g_1}(x)\pi_{g_2}(y)$. Hence, in the language of this $\Omega$-algebra, the last equation of \eqref{eqvariety} is equivalent to
$$
(\pi_g(x))^\ast-\pi_{g^\ast}(x^\ast)=0.
$$
Using either the absolutely free $G$-graded algebra or the (relatively) free $\Omega$-algebra, the free $G$-graded algebra with a $\ast$-involution $\mathbb{F}\langle X^G,\ast\rangle$ is the quotient of $\mathbb{F}\{X^G,\ast\}$ by the identities \eqref{eqvariety}.

As discussed in \cite{MFb}, in the special case where $\ast$ is a degree-preserving involution, then we can define the new variables $x_+^{(g)}:=x^{(g)}+(x^{(g)})^\ast$ and $x_-^{(g)}=x^{(g)}-(x^{(g)})^\ast$ (the symmetric and skew symmetric variables). Then we get the classical construction of the free (graded) $\ast$-algebra. Since $\ast$ does not necessarily preserve the homogeneous degree, we cannot use such technique in our context since these variables are not necessarily homogeneous.

Given a $G$-graded algebra $(\mathcal{A},\Gamma)$ with a homogeneous involution $\ast$, we denote by $\mathrm{Id}_G(\mathcal{A},\Gamma)$ its ideal of graded polynomial identities, and by $\mathrm{Id}_{G,\ast}(\mathcal{A},\Gamma,\ast)$ the set of all of its graded polynomial identities with involution.

\subsection{Gradings on $\mathrm{UT}_n$} The algebra $\mathrm{UT}_n=\mathrm{UT}_n(\mathbb{F})$ is the set of upper triangular matrices with entries in the field $\mathbb{F}$, that is,
$$
\mathrm{UT}_n=\left\{\left(\begin{array}{ccc}a_{11}&\cdots&a_{1n}\\&\ddots&\vdots\\0&&a_{nn}\end{array}\right)\mid a_{ij}\in\mathbb{F},\,i\le j\right\}.
$$
We denote by $e_{ij}$, $i\le j$ the matrix units, that is, the matrix having entry $1$ at $(i,j)$ and $0$ elsewhere.

Let $G$ be a group and $\eta=(g_1,\ldots,g_{n-1})\in G^{n-1}$. Then $\eta$ defines a $G$-grading on $\mathrm{UT}_n$ if we set $\deg e_{i,i+1}=g_i$. This kind of grading is called \emph{elementary}, and we shall use $\eta$ as well to denote it. It turns out that every group grading is isomorphic to an elementary grading (see the main result of \cite{VaZa2007}). Moreover, \cite[Proposition 1.6]{VinKoVa2004} tells us that two $G$-gradings defined by such sequences are isomorphic if and only if they are equal.

Now, let $\nu=(g_1,\ldots,g_m)\in G^m$ be another sequence. Following \cite[Definition 2.1]{VinKoVa2004}, we say that $\nu$ is $\eta$-good if there exist strictly upper triangular matrix units $r_1$, \dots, $r_m$ such that $\deg_\eta r_i = g_i$ and $r_1\cdots r_m\ne0$. If $\nu$ is not $\eta$-good, then it is called $\eta$-bad. The $\eta$-bad sequences are related to graded polynomial identities \cite[Proposition 2.2]{VinKoVa2004}, while $\eta$-good sequences tell us the way we can multiply variables of non-trivial degree. It is worth mentioning that the graded polynomial identities of $(\mathrm{UT}_n,\eta)$ may be constructed from $\eta$-bad sequences and graded polynomial identities of the base field (see \cite{VinKoVa2004,GRiva}).

\subsection{Involutions}
The involutions of the first kind on $UT_n$ were described in \cite{VKS} over fields of characteristic not $2$. If $n$ is odd each involution on $UT_n$ is equivalent to single  involution $\tau$, which we will call the orthogonal involution. It is given by reflection along the secondary diagonal, that is, for each $i\leq j$, $e_{i,j}^\tau = e_{n-j+1,n-i+1}$. If $n = 2m$ is even, any involution on $UT_n$ is equivalent either to the orthogonal involution or to the symplectic involution. If we denote it by $s$ then it is given by $A^s = D A^\tau D^{-1}$, where 
$D = \begin{pmatrix}
    I_m & 0 \\ 0 & -I_m
\end{pmatrix}$. 

In this paper we will consider the two involutions above on $UT_n$, namely  $\ast=\tau$ or $\ast=s$, the orthogonal or the symplectic involution on $\mathrm{UT}_n$. We denote $\varepsilon_\ast=1$ if $\ast=\tau$ and $\varepsilon_\ast=-1$ if $ * = s$. Note that $e_{1n}^\ast=\varepsilon_\ast e_{1n}$.

\subsection{Notations}

We assume that $\mathbb{F}$ is an arbitrary field, finite or not, of any characteristic.

By $G=\langle\alpha_1,\ldots,\alpha_{n-1}\rangle$ we denote the free group of rank $n-1$, freely generated by $\{\alpha_1,\alpha_2,\ldots,\alpha_{n-1}\}$. The group $G$ has an involution $\ast$ if we assume that $\alpha_i^\ast=\alpha_{n-i}$, for each $i=1,2,\ldots,n-1$. We denote by $\mathbb{F}\langle X^G,\ast\rangle$ the free associative $G$-graded algebra with a $\ast$-homogeneous involution $\ast$. The variables of trivial degree will be denoted by $x_i=x_i^{(1)}$. The variables of degree $g$ will be denoted by $x_i^{(g)}$. We use $x_i^{(g)\ast'}$ to denote either $x_i^{(g)}$ or $x_i^{(g)\ast}$. We use $z=x_i^{(g)\ast'}$ to denote a variable of non-trivial degree.

Let $\eta=(\alpha_1,\ldots,\alpha_{n-1})$, and consider the $\eta$-grading on $\mathrm{UT}_n$, that is $\deg_\eta e_{i,i+1}=\alpha_i$. We shall use $\eta$ to refer to the grading as well. 

\section{Polynomial identities}
We give a list of a few polynomial identities:
\begin{Lemma}\label{basicid}
$(\mathrm{UT}_n,\eta,\ast)$ satisfies the following polynomial identities:
\newcounter{bb}
\begin{enumerate}
\renewcommand{\labelenumi}{(\roman{enumi})}
\item $x^{(g)}$, $g\notin\mathrm{Supp}\,\eta$,
\item $x^{(g)\ast}-\varepsilon_\ast x^{(g)}$, where $g=\deg_\eta e_{ij}$, $i+j=n+1$,
\item $[x_1^{(1)},x_2^{(1)}]$.
\setcounter{bb}{\arabic{enumi}}
\end{enumerate}
In addition, if $\mathbb{F}$ is finite with $q$ elements, then it satisfies:
\begin{enumerate}
\renewcommand{\labelenumi}{(\roman{enumi})}
\setcounter{enumi}{\arabic{bb}}
\item $(x^{(1)})^q-x^{(1)}$.
\end{enumerate}
\end{Lemma}
\begin{proof}
The proof is a direct computation and it shall be omitted.
\end{proof}

If $n$ is odd then we have an extra set of polynomial identities:
\begin{Lemma}\label{extraextraid}
Let $n=2k+1$, $\deg_G z_1=\deg_\eta e_{k+1,j}$ and $\deg_G z_2=\deg_\eta e_{i,k+1}$, for some $i<k+1<j$, and $\deg_G x=1$. Then $(\mathrm{UT}_n,\eta,\ast)$ satisfies
$$
(x^\ast-x)z_1,\quad z_2(x^\ast-x).
$$
\end{Lemma}
\begin{proof}
The proof is a direct computation.
\end{proof}

Let $I$ be the $T_{G,\ast}$-ideal generated by the polynomials in \Cref{basicid} and, if $n$ is odd, \Cref{extraextraid} as well.

\begin{Def}\label{normalcom}
A \emph{normal monomial} is
$$
c=\omega_0 z_1\omega_1 z_2\cdots \omega_r z_r\omega_{r+1},
$$
where each $z_i$ is a variable of nontrivial degree (either $x^{(g)}$ or $x^{(g)\ast}$), the sequence $(\deg_G z_1,\ldots,\deg_G z_r)$ is $\eta$-good, and $\omega_j$ is
$$
x_{i_1}\cdots x_{i_r}x_{j_1}^\ast\cdots x_{j_s}^\ast,
$$
where $i_1\le\cdots\le i_r$, $j_1\le\cdots\le j_s$, $r, s\ge0$.
\end{Def}

\begin{Lemma}\label{basiclemma}
Modulo $I$, every element of $\mathbb{F}\langle X^G,\ast\rangle$ is a linear combination of normal monomials.
\end{Lemma}
\begin{proof}
The proof is standard.
\end{proof}

When the base field is finite, using the last polynomial identity, we can say something more on the normal monomials:
\begin{Lemma}\label{finitecase}
If $\mathbb{F}$ is finite having $q$ elements, then in each of the normal monomials of \Cref{basiclemma}, we may assume that for each $k$ and $l$,  $x_k$ and $x_l^\ast$ appears at most $q-1$ times each in each $\omega_j$.
\end{Lemma}
\begin{proof}
It follows from identity (iv) of \Cref{basicid}.
\end{proof}

Now, we shall prove that some of the $\ast$ may not appear in some situation.

\begin{Def}
We say that $g\in G$ is \emph{symmetric} if $g=\deg_\eta e_{ij}$, for some $i<j$ such that $i+j=n+1$. Given a sequence $(g_1,\ldots,g_m)$ of elements of $G$, we say that it has a \emph{symmetric subsequence} if there is $i<j$ such that $g_ig_{i+1}\cdots g_j$ is symmetric. In this case, the sequence $(g_i,g_{i+1},\ldots,g_j)$ is called a symmetric subsequence. A variable $z$ is called \emph{symmetric} if $\deg_\eta z$ is symmetric; and a monomial $z_1\cdots z_m$ is symmetric if $\deg_\eta z_1\cdots z_m$ is symmetric.
\end{Def}
\begin{Remark}
Recall that, if $g\in G$ is symmetric, then $x^{(g)\ast}-\varepsilon_\ast x^{(g)}\in I$.
\end{Remark}

\begin{Lemma}\label{extraid}
If $z$ is a symmetric variable, then, modulo $I$,
$$
zx=x^\ast z.
$$
\end{Lemma}
\begin{proof}
One has
$$
zx=\varepsilon_\ast(zx)^\ast=\varepsilon_\ast x^\ast z^\ast=x^\ast z.
$$
\end{proof}

Combining the previous results, we find restrictions on the normal monomials.
\begin{Lemma}\label{lem4}
Modulo $I$, every element of $\mathbb{F}\langle X^G,\ast\rangle$ is a linear combination of normal monomials $c=\omega_0z_1\omega_1\cdots\omega_rz_r\omega_{r+1}$, where:
\begin{enumerate}
\item if $(\deg_Gz_a,\ldots,\deg_G z_b)$ is a symmetric subsequence, then no $\ast$ appears in $\omega_{a-1}$ and $\omega_{b+1}$,
\item if $n=2k+1$, and for some $i$ and $j$, it holds that $\deg_G z_i=\deg_\eta e_{k+1,j}$, then there is no $\ast$ in $\omega_{i-1}$,
\item if $n=2k+1$, and for some $i$ and $j$, it holds that $\deg_G z_i=\deg_\eta e_{j,k+1}$, then there is no $\ast$ in $\omega_{i}$.
\end{enumerate}
In addition, if $\mathbb{F}$ is finite having $q$ elements, then for each $j$, the number of $x_j$ and the number of $x_j^\ast$ appearing in each $\omega_i$ are at most $q-1$.
\end{Lemma}
\begin{proof}
The proof is a direct consequence of the previous lemmas.
\end{proof}

Now, we shall investigate the distinct elements of the relatively free algebra that are the product of variables of non-trivial degree. We prove that two monomials containing only non-trivial degree variables are equal (up to a scalar) if and only if one of them is obtained from the other by successive applications of $\ast$ in symmetric subsequence of the variables.
\begin{Lemma}\label{lem5}
Let $z_1=x_1^{(g_1)\ast'}$, \dots, $z_m=x_m^{(g_m)\ast'}$, where $g_i\ne1$, for all $i=1,2,\ldots,m$, where $\ast'$ means either $\ast$ or nothing. Let $z=z_1\cdots z_m$ and $z'=z_{\sigma(1)}^{\ast'}\cdots z_{\sigma(m)}^{\ast'}$, for some $\sigma\in\mathcal{S}_m$. Then, $z=\pm z'$ modulo $\mathrm{Id}_{G,\ast}(\mathrm{UT}_n,\eta,\ast)$ if and only if $z=\pm z'$ modulo $I$.
\end{Lemma}
\begin{proof}
The proof will be by induction on the total degree of $z$, that is, on $m$. There is nothing to do if $m=1$, so we may assume that $m>1$. First, we assume that $z$ is symmetric. Note that two variables of non-trivial degree are linearly dependent modulo $\mathrm{Id}_{G,\ast}(\mathrm{UT}_n,\eta,\ast)$ if and only if their $G$-homogeneous degree coincide. If $z+\mathrm{Id}_{G,\ast}(\mathrm{UT}_n,\eta,\ast)=\pm z'+\mathrm{Id}_{G,\ast}(\mathrm{UT}_n,\eta,\ast)$, then
$$
g_{\sigma(1)}^{\ast'}\cdots g_{\sigma(m)}^{\ast'}=g_1\cdots g_m=\alpha_a\alpha_{a+1}\cdots\alpha_b.
$$
Thus, $g_{\sigma(1)}$ must start with $\alpha_a$. It means that either $g_{\sigma(1)}=g_1$, or $g_{\sigma(1)}=g_1^\ast$. The first equality implies that $g_2\cdots g_m=g_{\sigma(2)}^{\ast'}\cdots g_{\sigma(m)}^{\ast'}$, so $z_2\cdots z_m+\mathrm{Id}_{G,\ast}(\mathrm{UT}_n,\eta,\ast)=\pm z_{\sigma(2)}^{\ast'}\cdots z_{\sigma(m)}^{\ast'}+\mathrm{Id}_{G,\ast}(\mathrm{UT}_n,\eta,\ast)$ and the result follows by induction. For the latter, since $g$ is symmetric, $z=\pm\varepsilon_\ast(z')^\ast$, so it returns to the previous case.

Now, assume that $z$ is not necessarily symmetric. Repeating the notation of the beginning of the argument of the previous paragraph, we obtain that $g_{\sigma(1)}=g_p^\ast$, for some $p\in\{2,\ldots,m-1\}$ (the case $p=1$ is already done, and $p=m$ means that $z$ is symmetric). Writing $g_p=\alpha_u\cdots\alpha_{v}$, it means that $a+v=n$. In particular, for any $k\in\{a,a+1,\ldots,b\}$, we have $b+k>v+a=n$. That is, there is no index $k$ such that $\alpha_k^\ast=\alpha_b$; so there is no index $t\in\{1,\ldots,m\}$ such that $g_t^\ast =g_b$. As a consequence, $g_{\sigma(m)}^{\ast'}=g_m$. Thus, $z_1\cdots z_{m-1}+\mathrm{Id}_{G,\ast}(\mathrm{UT}_n,\eta,\ast)=\pm z_{\sigma(1)}^{\ast'}\cdots z_{\sigma(m-1)}^{\ast'}+\mathrm{Id}_{G,\ast}(\mathrm{UT}_n,\eta,\ast)$, and the result follows by induction.
\end{proof}

As a consequence, we obtain the following characterization.
\begin{Cor}\label{cor}
Let $z_1=x_1^{(g_1)\ast'}$, \dots, $z_m=x_m^{(g_m)\ast'}$, where $g_i\ne1$, for all $i=1,2,\ldots,m$, where $\ast'$ means either $\ast$ or nothing. Then, a subset of
$$
\left\{z_{\sigma(1)}^{\ast'}\cdots z_{\sigma(m)}^{\ast'}\mid\sigma\in\mathcal{S}_m\right\}
$$
is linearly independent modulo $I$ if and only if it is linearly independent modulo $\mathrm{Id}_{G,\ast}(\mathrm{UT}_n,\eta,\ast)$.
\end{Cor}
\begin{proof}
Let $S$ be a subset, which is linearly independent modulo $I$. It is enough to prove that the subset $S_g=\{m\in S\mid\deg_G m=g\}$ is linearly independent modulo $\mathrm{Id}_{G,\ast}(\mathrm{UT}_n,\eta,\ast)$, for each $g\in G$. Let $i<j$ be such that $\deg_\eta e_{ij}=g$. Since any evaluation of elements of $S_g$ by $(\mathrm{UT}_n,\eta,\ast)$ gives a multiple of $e_{ij}$, we see that two elements of $S_g$ are equal modulo $\mathrm{Id}_{G,\ast}(\mathrm{UT}_n,\eta,\ast)$. From \Cref{lem5}, it gives that two elements of $S_g$ are equal modulo $I$. Since $S_g$ is linearly independent modulo $I$, it shows that each $S_g$ contains at most $1$ element. Thus, $S$ is linearly independent modulo $\mathrm{Id}_{G,\ast}(\mathrm{UT}_n,\eta,\ast)$. The converse is immediate.
\end{proof}

Finally, we have the last key lemma:
\begin{Lemma}\label{keylemma}
Let $\beta$ be a subset of monomials consisting of products of variables of nontrivial degree, such that $\{m+I\mid m\in \beta \}$ is a basis of
$$
\{x_{i_1}^{(g_1)\ast'}\cdots x_{i_s}^{(g_s)\ast'}+I\mid s>0, (g_1,\ldots,g_s)\text{ is $\eta$-good}\}.
$$
Then, the set of normal monomials
$$
\omega_0z_1\omega_1\cdots z_t\omega_t,\quad z_1\cdots z_t\in\beta,
$$
satisfying the conditions of \Cref{lem4}, is a basis of $\mathbb{F}\langle X^G,\ast\rangle$, modulo $\mathrm{Id}_{G,\ast}(\mathrm{UT}_n,\eta,\ast)$.
\end{Lemma}
\begin{proof}
We consider a set $S$ of monomials, each of them having a same fixed set of variables appearing in all of them. Denote by $x_1$, \dots, $x_k$ the variables of trivial degree appearing in the above set. We fix an order and denote by $z_1$, \dots, $z_t$ the variables of non-trivial degree (having $\ast$ or not, and possibly with repetition of the kind $z_i=x^{(g)}$ and $z_j=x^{(g)\ast}$). The set of monomials in $S$ may be supposed to have the same $G$-degree. From \Cref{cor}, it implies that we may assume that the leading variables of the normal monomials are $z_1$, \dots, $z_t$ and in this order.

For the $z_i$ there is a unique possibility for an evaluation of the kind $e_{p_iq_i}$. Next, we consider the polynomial algebra $\mathbb{F}[\xi_{i,\ell}]$ (in commutative variables), and the evaluation
$$
x_\ell=\sum_{j=1}^n\xi_{j,\ell}e_{jj},\quad\ell=1,\ldots,k.
$$
The evaluation gives $e_{p_1,q_t}$ multiplied by some polynomial in the commuting variables $\xi_{j,\ell}$. For each of the monomials, assume that the variable $x_\ell$ appears in the commutator $\omega_i$. Then, it contributes with the factor $\xi_{\ell,j_i}$ if $x_\ell$ appears without $\ast$, and $\xi_{\ell,n-j_k+1}$ otherwise; or $\xi_{\ell,i_1}$ or $\xi_{\ell,n-i_1+1}$ if in $\omega_0$.

Now, we shall prove that the factor $\xi_{\ell,k}$ uniquely determines the position of the variable of trivial degree. So, assume that $i_p=n-j_q+1$, and suppose that $p<q$. Then $i_p+j_q=n+1$. Thus, $(z_p,\ldots,z_q)$ is symmetric, so there is no $x_\ell^\ast$ that appears in any of $\omega_{p-1}$ or $\omega_q$.

Finally, note that if $\mathbb{F}$ is finite having $q$ elements, then $\xi^q-\xi$ is a polynomial identity for $\mathbb{F}$. However, from the last assertion of \Cref{lem4}, there is no $q$ power of a variable appearing in an evaluation. The result is proved.
\end{proof}

As a consequence, we obtain the first main result of the paper:
\begin{Thm}\label{polid}
Let $\mathbb{F}$ be an arbitrary field and $n\in\mathbb{N}$. Let $G=\langle\alpha_1,\ldots,\alpha_{n-1}\rangle$ be the free group of rank $n-1$, and $\eta=(\alpha_1,\ldots,\alpha_{n-1})$ define an elementary $G$-grading on $\mathrm{UT}_n$. Let $\ast$ be a homogeneous $\eta$-involution on $\mathrm{UT}_n$, and denote $\varepsilon_\ast=1$ if $\ast$ is orthogonal and $\varepsilon_\ast=-1$ otherwise. Then $\mathrm{Id}_{G,\ast}(\mathrm{UT}_n,\eta,\ast)$ follows from:
\begin{enumerate}
\renewcommand{\labelenumi}{(\roman{enumi})}
\item $x^{(g)}$, $g\notin\mathrm{Supp}\,\eta$,
\item $x^{(g)\ast}-\varepsilon_\ast x^{(g)}$, where $g=\deg_\eta e_{ij}$, $i+j=n+1$,
\item $[x_1^{(1)},x_2^{(1)}]$.
\setcounter{bb}{\arabic{enumi}}
\end{enumerate}
If $n=2k+1$ is odd, then we have the extra identities:
\begin{enumerate}
\renewcommand{\labelenumi}{(\roman{enumi})}
\setcounter{enumi}{\arabic{bb}}
\item $(x^{(1)\ast}-x^{(1)})z$, where $\deg_Gz=\deg_\eta e_{k+1,i}$, for some $i$.
\setcounter{bb}{\arabic{enumi}}
\end{enumerate}

In addition, if $\mathbb{F}$ is finite with $q$ elements, then we shall include:
\begin{enumerate}
\renewcommand{\labelenumi}{(\roman{enumi})}
\setcounter{enumi}{\arabic{bb}}
\item $(x^{(1)})^q-x^{(1)}$.
\end{enumerate}

Moreover, a basis of $\mathbb{F}\langle X^G,\ast\rangle$, modulo $\mathrm{Id}_{G,\ast}(\mathrm{UT}_n,\eta,\ast)$, constitutes of all the polynomials described in \Cref{keylemma}.\qed
\end{Thm}

\section{On codimension sequence}

\subsection{Basic definitions}
We recall some definitions concerning codimension sequence. Let $G$ be a group and consider a $G$-grading $\Gamma$ on an algebra $\mathcal{A}$. Given a sequence $\mu=(g_1,\ldots,g_m)\in G^m$, we let
$$
P_\mu=\mathrm{Span}\left\{x_{\sigma(1)}^{(g_{\sigma(1)})}\cdots x_{\sigma(m)}^{(g_{\sigma(m)})}\mid\sigma\in\mathcal{S}_m\right\},\quad P_\mu(\mathcal{A},\Gamma)=P_\mu/P_\mu\cap\mathrm{Id}_G(\mathcal{A},\Gamma),
$$
where $\mathcal{S}_m$ is the symmetric group on the set of $m$ elements. The \emph{graded codimension sequence} of $(\mathcal{A},\Gamma)$ is
$$
c_m(\mathcal{A},\Gamma)=\dim\sum_{\mu\in G^m}P_\mu(\mathcal{A},\Gamma),\quad m\in\mathbb{N}.
$$
Now, assume that $\ast$ is a homogeneous involution on $(\mathcal{A},\Gamma)$. We denote
$$
P^{(\ast)}_\mu=\mathrm{Span}\left\{\left(x_{\sigma(1)}^{(g_{\sigma(1)})}\right)^{\ast'}\cdots\left(x_{\sigma(m)}^{(g_{\sigma(m)})}\right)^{\ast'}\mid\sigma\in\mathcal{S}_m\right\},
$$
where $\ast'$ means either $\ast$ or nothing.

As before, we let
$$
P_\mu^{(\ast)}(\mathcal{A},\Gamma,\ast)=P_\mu^{(\ast)}/P_\mu^{(\ast)}\cap\mathrm{Id}_{G,\ast}(\mathcal{A},\Gamma,\ast),
$$
and the $(G,\ast)$-codimension sequence is defined by
$$
c_m(\mathcal{A},\Gamma,\ast)=\dim\sum_{\mu\in G^m}P_\mu^{(\ast)}(\mathcal{A},\Gamma,\ast).
$$
The graded exponent and the $(G,\ast)$-exponent (if exists) are respectively defined by
$$
\mathrm{exp}(\mathcal{A},\Gamma)=\lim\limits_{m\to\infty}\sqrt[m]{c_m(\mathcal{A},\Gamma)},\quad\mathrm{exp}(\mathcal{A},\Gamma,\ast)=\lim\limits_{m\to\infty}\sqrt[m]{c_m(\mathcal{A},\Gamma,\ast)}.
$$
Sometimes we shall use the following notation. Given a sequence $\mu=(z_1,\ldots,z_m)$ of variables, where $z_i=x_{j_i}^{(g_i)\ast'}$, then we let
$$
P_\mu=\mathrm{Span}\{z_{\sigma(1)}\cdots z_{\sigma(m)}\mid\sigma\in\mathcal{S}_m\}.
$$

\subsection{Upper triangular matrices} We keep our notation of $(\mathrm{UT}_n,\eta,\ast)$, where $G$ denotes the free group of rank $n-1$. In this section, we assume that $\mathbb{F}$ has characteristic zero.

\noindent\textbf{Notation.} Let $z_1=x_1^{(g_1)\ast'}$, \dots, $z_m=x_m^{(g_m)\ast'}$ be variables of nontrivial degree. We denote
$$
Q(z_1,\ldots,z_m,x_1,\ldots,x_\ell)=\sum P(z_1,\dots,z_m,x_1^{\ast'},\ldots,x_\ell^{\ast'}),
$$
that is, the set of all multilinear polynomials of degree $m+\ell$, where the set of variables of nontrivial degree is $\{z_1,\ldots,z_m\}$.

We compute the asymptotic behaviour of the codimension sequence of $(\mathrm{UT}_n,\eta,\ast)$. For, we split the discussion in a few steps.
\begin{Lemma}\label{codim}
For any $m\in\mathbb{N}$,
$$
c_m^{(G,\ast)}(\mathrm{UT}_n,\eta,\ast)=\sum_{\ell=0}^{\min\{m,n-1\}}\binom{m}{\ell}\ell!\omega(\ell,m),
$$
where $\omega(\ell,m)=\dim\sum Q(x_1^{(g_1)\ast'},\ldots,x_\ell^{(g_\ell)\ast'},x_1,\ldots,x_{m-\ell})$, and the summation runs over all sequences $(g_1,\ldots,g_\ell)$ of nontrivial elements.
\end{Lemma}

\begin{proof}
We let $\mu=(u_1,\ldots,u_m)$, where each $u_i=x_i^{(h_i)\ast'}$ (where, as usual, $\ast'$ means either $\ast$ or nothing) be a sequence, and consider
$$
P_\mu(\mathrm{UT}_n,\eta,\ast)=\mathrm{Span}\left\{u_{\sigma(1)}\cdots u_{\sigma(m)}+\mathrm{Id}_{G,\ast}(\mathrm{UT}_n,\eta,\ast)\mid\sigma\in\mathcal{S}_m\right\}.
$$
It is enough to compute $\dim\sum_\mu P_\mu(\mathrm{UT}_n,\eta,\ast)$, where the sum runs over all sequences $\mu$ of all variables, indexed by $1$, \dots, $m$, and all possibilities of $G$-degree, and with or without $\ast$. Denote $g_i=\deg_\eta z_i$, and assume that $(g_1,\ldots,g_m)$ contains exactly $\ell$ elements of non-trivial degree. Note that $\ell\in\{0,1,\ldots,n-1\}$.

First, there are $\binom{m}{\ell}$ ways to distribute the lower index for the variables of non-trivial degree. Then, you distribute the indexes between the $\ell$ variables, obtaining $\ell!$ possibilities. Next, for each $\eta$-good sequence of length $\ell$, we need to count the number of linearly independent normal monomials (\Cref{normalcom}) in $Q(z_1,\ldots,z_\ell,x_1,\ldots,x_{m-\ell})$ satisfying \Cref{lem4}. This number is denoted by $\omega(\ell,m)$.
\end{proof}

We can compute the exact value of $\omega(\ell,m)$ when $\ell=n-1$. For the other cases, we shall obtain an upper bound for $\omega(\ell,m)$. For, we need to find the symmetric subsequences.

If $\nu=(z_1,\ldots,z_{n-1})$ is a sequence of variables such that their respective degrees $(\deg_\eta z_1,\ldots,\deg_\eta z_{n-1})$ is $\eta$-good, then necessarily the sequence is uniquely determined, namely $\nu=(x_1^{(\alpha_1)},\ldots,x_{n-1}^{(\alpha_{n-1})})$. So, every subsequence of $\nu$ centered in the middle is symmetric. Thus, we obtain:
\begin{Lemma}\label{case:nminusone}
If $n>1$, then $\displaystyle\omega(n-1,m)=2^{\lfloor\frac{n-1}2\rfloor}n^{m-n+1}$.
\end{Lemma}
\begin{proof}
Since there is a single $\eta$-good sequence of length $n-1$, we know that $z_i=x_i^{(\alpha_i)}$ or $z_i=x_i^{(\alpha_{n-i})\ast}$. Now, since $(\alpha_j,\alpha_{j+1},\ldots,\alpha_{n-j})$ is symmetric for all $j=1,2,\ldots,\lfloor\frac{n-1}2\rfloor$, we may apply $\ast$ to the respective symmetric subsequence, obtaining a multiple of itself. If $n$ is even, then $(\alpha_{\frac{n}2})$ is symmetric, so we may remove the $\ast$ from $z_{\frac{n}2}$. Hence, we may assume that the $\ast$ does not appear on all $z_j$ where $j>\lfloor\frac{n-1}2\rfloor$. Thus, there are $2^{\lfloor\frac{n-1}2\rfloor}$ ways to choose between the variables of non-trivial degree. We need to compute the number of polynomials in $$Q(z_1,\ldots,z_{n-1},x_1,\ldots,x_{m-n+1}),$$ satisfying \Cref{lem4}. Thus, from the previous discussion, all the variables of trivial degree appears without $\ast$. There are $n$ possible positions for each variable of trivial. Thus, the number of such polynomials coincide with the number of ways of putting $m-(n-1)$ distinct balls in $n$ boxes. This total number equals $n^{m-n+1}$.
\end{proof}

Now, we shall obtain an upper bound for the numbers $\omega(\ell,m)$, where $\ell<n-1$. We start with:
\begin{Lemma}\label{firstupper}
If $\displaystyle\ell\le\left\lfloor\frac{n-1}2\right\rfloor$, then
$$
\omega(\ell,m)\le2^m\binom{n}{\ell+1}(\ell+1)^{m-\ell}.
$$
In particular, if $n$ is even or $\ell<\lfloor\frac{n-1}2\rfloor$, then $\omega(\ell,m)\le2^\ell\binom{n}{\ell+1}n^{m-\ell}$.
\end{Lemma}
\begin{proof}
There is a total of $\binom{n}{\ell+1}$ $\eta$-good sequences of length $\ell+1$. Indeed, this number equals the number of nonzero products of strict upper triangular matrix units $e_{i_1j_1}e_{i_2j_2}\cdots e_{i_\ell j_{\ell}}$. The former is nonzero if and only if
$$
1\le i_1<j_1=i_2<j_2=\cdots =i_\ell<j_\ell\le n.
$$
Hence, this is the number of ways of choosing $\ell+1$ distinct numbers in the set $\{1,2,\ldots,n\}$. This last number is exactly $\binom{n}{\ell+1}$. Given $\ell$ variables of non-trivial degree, each of them may appear with $\ast$ or not; obtaining at most $2^\ell$ possibilities.

Now, fixed an $\eta$-good sequence of length $\ell$, we are left with $m-\ell$ variables of trivial degree. Each of them may appear in some of the $\omega_i$; so there is a total of $(\ell+1)^{m-\ell}$ possibilities. Now, each of the variables may or may not contain $\ast$, totaling at most $2^{m-\ell}$ possibilities. Hence,
$$
\omega(\ell,m)\le\binom{n}{\ell+1}2^\ell2^{m-\ell}(\ell+1)^{m-\ell}=2^m\binom{n}{\ell+1}(\ell+1)^{m-\ell}.
$$
Since $2\ell\le2\lfloor\frac{n-1}2\rfloor\le n-1$, where the inequality is strict if $n$ is even, we get the last assertion.
\end{proof}

Now, we need a technical result:
\begin{Lemma}\label{keylemmacodim}
Assume that $n>1$ and let $\mu=(g_1,\ldots,g_\ell)$ be an $\eta$-good sequence. If $\displaystyle\ell\ge\left\lfloor\frac{n-1}2\right\rfloor$, then $\mu$ contains at least $\ell-\lceil\frac{n}2\rceil+1$ symmetric subsequences. Additionally, if $n$ is odd, then either we may increase 1 extra symmetric subsequence or $\mu$ contains a variable of degree $\deg_\eta e_{k+1,j}$ or $\deg_\eta e_{j,k+1}$.
\end{Lemma}
\begin{proof}
Consider a stand having $\lceil\frac{n}2\rceil$ shelves. If $n$ is even, then each shelf has exactly two numbered slots that can accommodate a ball. The numbering is as follows: the first lower shelf contains the left slot $1$ and right slot $n$, the second shelf above the first has the left slot $2$ and right slot $n-1$, and so on. If $n$ is odd, then there are $\frac{n-1}{2}$ shelves with two numbered slots as before, and a last upper shelf with a single slot with the number $\frac{n+1}2$, and we shall call it a center slot (that is, it is not a right nor left slot).

Now, we put $\ell+1$ balls in the slots. Assume that the balls lies in the slots with (ordered) numbers $(i_1,i_2,\ldots,i_{\ell+1})$. Then, the positions of the balls define the $\eta$-good sequence $(\deg_\eta e_{i_1i_2},\deg_\eta e_{i_2i_3},\ldots,\deg_\eta e_{i_\ell i_{\ell+1}})$, and conversely, each $\eta$-good sequence may be represented by putting balls in the adequate slots. Set $z_j=x^{(\deg_\eta e_{i_ji_{j+1}})}$, for each $j=1,2,\ldots,\ell$.

Now, note that the number of symmetric subsequences of $(z_1,\ldots,z_\ell)$ is exactly the number of fully completed shelves. Additionally, if $n=2k+1$, then $(z_1,\ldots,z_\ell)$ contains a variable of degree $\deg_\eta e_{k+1,j}$ or $\deg_\eta e_{j,k+1}$ if and only if the upper shelf has a ball. So the result follows from pigeonhole principle.
\end{proof}

\begin{Lemma}\label{secondupper}
If $\displaystyle\ell\ge\left\lfloor\frac{n-1}2\right\rfloor$, then we have
$$
\omega(\ell,m)\le2^\ell\binom{n}{\ell+1}n^{m-\ell}.
$$
\end{Lemma}
\begin{proof}
The begining of the proof is the same as the first paragraph of the proof of \Cref{firstupper}. So, let us fix an $\eta$-good sequence of length $\ell$. Every variable of nontrivial degree may or may not have an $\ast$, resulting in at most $2^\ell$ possibilities. As before, we have $m-\ell$ variables of trivial degree. Let $x$ be a variable of trivial degree. From \Cref{keylemmacodim}, there are at least $\ell-\lceil\frac{n}2\rceil+1$ symmetric subsequences. It means that, in at least $2(\ell-\lceil\frac{n}2\rceil+1)$ positions, $x$ appears without an $\ast$; and the remaining positions it may appear as $x$ or $x^\ast$. For, we have at most
$$
\left(2(\ell-\left\lceil\frac{n}2\right\rceil+1)+2(\ell+1-2(\ell-\left\lceil\frac{n}2\right\rceil+1))\right)^{m-\ell}
$$
possibilities. As a consequence, if $n$ is even, then
$$
\omega(\ell,m)\le2^\ell\binom{n}{\ell+1}\left(2\left\lceil\frac{n}2\right\rceil\right)^{m-\ell}=2^\ell\binom{n}{\ell+1}n^{m-\ell}.
$$
If $n$ is odd, then \Cref{keylemmacodim} gives a sharper estimate. In comparasion of the even case, there is at least one extra position where $x$ appears without an $\ast$. Then, we obtain
$$
\left(2(\ell-\left\lceil\frac{n}2\right\rceil+1)+1+2(\ell+1-2(\ell-\left\lceil\frac{n}2\right\rceil+1)-1)\right)^{m-\ell}
$$
at most possibilities to distribute the variables of trivial degree. This gives
$$
\omega(\ell,m)\le2^\ell\binom{n}{\ell+1}\left(2\left\lceil\frac{n}2\right\rceil-1\right)^{m-\ell}=2^\ell\binom{n}{\ell+1}n^{m-\ell}.
$$
The result is complete.
\end{proof}

Given two maps $f,g:\mathbb{N}\to\mathbb{N}$, we denote $f\sim g$ if $\lim_{n\to\infty}f(n)/g(n)=1$. As a consequence, we obtain the following result:
\begin{Thm}\label{codimThm}
Let $\mathbb{F}$ be a field of characteristic $0$ and $n>1$. Then the asymptotic behaviour of the $(G,\ast)$-codimension sequence of $(\mathrm{UT}_n,\eta,\ast)$ is
$$
c_m(\mathrm{UT}_n,\eta,\ast)\sim\frac{2^{\lfloor\frac{n-1}2\rfloor}}{n^{n-1}}m^{n-1}n^m
$$
In particular, $\mathrm{exp}(\mathrm{UT}_n,\eta,\ast)=n$.
\end{Thm}
\begin{proof}
Let $k=\lfloor\frac{n-1}2\rfloor$. We use the expression of $c_m(\mathrm{UT}_n,\eta,\ast)$ computed in \Cref{codim}. It is enough to prove that, for each $\ell=0,1,\ldots,n-1$,
$$
\lim_{n\to\infty}\frac{\binom{m}{\ell}\ell!\omega(\ell,m)}{2^km^{n-1}n^{m-n+1}}=\delta_{\ell,n-1}.
$$
We shall deal with the case $\ell<n-1$ first. Using either \Cref{firstupper} or \Cref{secondupper}, we obtain
$$
\lim_{m\to\infty}\frac{\binom{m}{\ell}\ell!\omega(\ell,m)}{2^km^{n-1}n^{m-n+1}}\le\lim_{m\to\infty}\frac{\binom{m}{\ell}\ell!2^\ell\binom{n}{\ell+1}n^{m-\ell}}{2^km^{n-1}n^{m-n+1}}=0.
$$
Now, from the value of $\omega(n-1,m)$ computed in \Cref{case:nminusone}, we obtain
$$
\lim_{m\to\infty}\frac{\binom{m}{n-1}(n-1)!2^{k}n^{m-n+1}}{2^km^{n-1}n^{m-n+1}}=1.
$$
The result is proved.
\end{proof}
\begin{Remark}
If $n=1$, then, modulo the identities, we are left with commutative variables $x_i$ satisfying $x_i^\ast=x_i$. It means that
$$
c_m(\mathrm{UT}_1(\mathbb{F}),\eta,\ast)=1,\quad\forall m\in\mathbb{N}.
$$
In particular, $\mathrm{exp}(\mathrm{UT}_1(\mathbb{F}),\eta,\ast)=1$.
\end{Remark}

\subsection{Other gradings and involution} Now, let $H$ be any group and consider an elementary $H$-grading $\Gamma$ on $\mathrm{UT}_n$, and assume that $\ast$ is a homogeneous involution on $(\mathrm{UT}_n,\Gamma)$. Then, there is group homomorphism $G\to H$ such that $\psi:\alpha_i\mapsto\deg_\Gamma e_{i,i+1}$, for $i=1,2,\ldots,n-1$. So $\psi$ coalesce the grading $\eta$ obtaining $\Gamma$. Thus, the $(H,\ast)$-algebra $(\mathrm{UT}_n,\Gamma,\ast)$ is obtained from $(\mathrm{UT}_n,\eta,\ast)$ by a coarsening of the grading. Hence, $(\mathrm{UT}_n,\eta,\ast)$ is, in some sense, the generator of the structures of star-graded-algebra on $\mathrm{UT}_n$.

In particular, the following results are valid, and the arguments are standard (see \cite{BahGR98}):
$$
c_m(\mathrm{UT}_n)\le c_m(\mathrm{UT}_n,\Gamma)\le c_m(\mathrm{UT}_n,\Gamma,\ast)\le c_m(\mathrm{UT}_n,\eta,\ast).
$$
Moreover,
$$
c_m(\mathrm{UT}_n)\le
c_m(\mathrm{UT}_n,\ast)\le c_m(\mathrm{UT}_n,\Gamma,\ast).
$$
Now, combining \cite[Corollary 13]{KY} and \Cref{codimThm}, we obtain:
\begin{Thm}\label{expsecexp}
Let $\mathbb{F}$ be a field of characteristic zero, and $(\mathrm{UT}_n,\Gamma,\ast)$ be $\mathrm{UT}_n$ endowed with a $H$-grading $\Gamma$ and a homogeneous involution $\ast$. Then $\mathrm{exp}(\mathrm{UT}_n,\Gamma,\ast)=n$. Moreover,
$$
\lim_{m\to\infty}\log_m\left(\frac{c_m(\mathrm{UT}_n,\Gamma,\ast)}{\mathrm{exp}(\mathrm{UT}_n,\Gamma,\ast)^m}\right)=n-1.
$$\qed
\end{Thm}

\end{document}